\documentclass[noinfoline, 12pt]{imsart} 

\usepackage[utf8]{inputenc} 

\usepackage{geometry} 
\usepackage{amsthm, amsfonts, amssymb, amsmath,enumitem, natbib, bbm, mathabx}
\usepackage{mathtools}
\mathtoolsset{showonlyrefs}
\newtheorem{defi}{Definition}[section]
\newtheorem{thm}{Theorem}[section]
\newtheorem{lem}{Lemma}[section]

\newtheorem{rem}{Remark}[section]
\providecommand{\norm}[1]{\left\lVert#1\right\rVert}
\DeclareMathOperator*{\argmin}{arg\,min}

\usepackage{graphicx} 
\usepackage{color, float}

\usepackage{booktabs} 
\usepackage{array} 
\usepackage{verbatim} 
\usepackage{subfig} 
\usepackage[colorlinks=true, pdfstartview=FitV,
linkcolor=blue, citecolor=blue, urlcolor=blue]{hyperref}
\makeatletter
\def\namedlabel#1#2{\begingroup
    #2%
    \def\@currentlabel{#2}%
    \phantomsection\label{#1}\endgroup
}
\makeatother
\newcommand{\vertiii}[1]{{\left\vert\kern-0.25ex\left\vert\kern-0.25ex\left\vert #1
    \right\vert\kern-0.25ex\right\vert\kern-0.25ex\right\vert}}

\usepackage[nottoc,notlof,notlot]{tocbibind} 
\usepackage[titles,subfigure]{tocloft} 

\begin{document}
 \begin{frontmatter}
\title{Model-free Study of Ordinary Least Squares Linear Regression}
\runtitle{Assumption-lean Linear Regression}
\begin{aug}
  \author{\fnms{Arun K.}  \snm{Kuchibhotla}\ead[label=e1]{arunku@wharton.upenn.edu}},
  \author{\fnms{Lawrence D.} \snm{Brown}},
  \and
  \author{\fnms{Andreas} \snm{Buja}}

  \runauthor{Kuchibhotla et al.}

  \affiliation{University of Pennsylvania}

  \address{University of Pennsylvania\\ \printead{e1}}

\end{aug}
\begin{abstract}
Ordinary least squares (OLS) linear regression is one of the most basic statistical techniques for data analysis. In the main stream literature and the statistical education, the study of linear regression is typically restricted to the case where the covariates are fixed, errors are mean zero Gaussians with variance independent of the (fixed) covariates. Even though OLS has been studied under misspecification from as early as the 1960's, the implications have not yet caught up with the main stream literature and applied sciences. The present article is an attempt at a unified viewpoint that makes the various implications of misspecification stand out. 
\end{abstract}
\end{frontmatter}
\section{Introduction and Motivation}
The aim of this article is to provide what we call an ``upside down	analysis'' for linear regression.  While traditional linear regression	analysis starts with assumptions such as fixed covariates as well as	linearity and Gaussian errors, upside down analysis starts with a	given estimator -- OLS in this case -- and finds the most general	conditions under which the estimator ``works'' in the sense that it	has a well-defined target and permits inference.  In our upside down	analysis, essentially all we need is a form of law of large numbers
(LLN) and a central limit theorem (CLT) for second moments of the	response and the covariates.  Such LLNs and CLTs are satisfied in	numerous situations, including strong mixing random variables,
martingales, Markov chains, time series processes,~\ldots (see, e.g.,
chapters 3 and 5 of \cite{White2001}).  LLNs and CLTs can accommodate	non-identical distributions of random vectors, a fact that turns out	to be a particularly useful feature of the proposed analysis: It	allows a treatment of fixed and random covariates in a unified way by	thinking of fixed values of covariates as degenerate point mass	distributions.

It should be mentioned here that most of the results presented in this	article are known in the literature but are scattered.  A unified	treatment as given in this article appears to be non-existent.
Somewhat close in spirit but executing a traditional ``upside up''
analysis is by \cite{White1980} who studies linear regression under the assumption of independence	allowing for non-identical distributions.  Our analysis sidesteps his	assumption of absent correlation between covariates and errors.  \cite{Gallant1988} and \cite{White2001} extend the analysis of \cite{White1980} to certain dependence structures	but remain traditional in that they define targets of estimation in	terms of asymptotic limits for a fixed number of covariates, whereas	we define sample size-dependent targets and allow the number of	covariates to grow.

An essential difference of our ``upside down'' approach to these	traditional treatments is that the latter assume the existence of a	single target such that certain conditions are satisfied.  For	example, the traditional linear model assumes there exists a
$\mathbf{\beta}_0$ such that $Y_i = X_i^{\top} \beta_0 + \varepsilon_i$ and
$\varepsilon_i \sim N(0,\sigma^2)$ iid, or, as in \cite{White1980},
$\mathbb{E}[X_i \varepsilon_i] = 0$ for all $i=1,\ldots, n$.  In contrast,
we make no such assumptions; rather, we construct {\em sequences of
  targets} for the OLS procedure that are {\em intrinsic to OLS}
without postulating a single target that is extraneous to the	procedure.  This is crucially possible by postulating LLNs and CLTs	for the components of the normal equations (estimating equations).

The remainder of the article is organized as follows. Section \ref{sec:Target}
describes the concept of ``target of estimation'' and provides the	minimal assumptions under which the least squares estimator
``works''. Even though the definition of target can be done under very	minimal assumptions, it is hard to proceed further to inference. For this reason, we add an assumption of	independence of observations to proceed. In Section~\ref{sec:FixedX},
the problem is studied under the only assumption of fixed covariates	and none of the other classical assumptions as mentioned above. In	Section~\ref{sec:RandomX}, the problem is studied under the assumption	that the observations are independent and identically distributed	random vectors. After a preliminary understanding of the problem in	both fixed and random covariates, a unified framework is developed for	the problem in Section~\ref{sec:Unified} along with a normal	approximation. To do inference (or more specifically construct confidence	intervals), a ``good'' variance estimator is needed.  Section~\ref{sec:VarEst} provides theory about ``asymptotic'' variance	estimation and also bootstrap based variance estimation. Section~\ref{sec:tAndFTests} considers the problem of testing hypotheses about the	target of estimation. In Section~\ref{sec:DeteministicInequality}, we summarize the discussion of the unified framework by providing a deterministic inequality for the least squares linear regression estimator which reassures that only CLT and LLN are required for linear regression estimator to work. We end this article with some	concluding remarks in Section~\ref{sec:Conclude}. A ``non-technical''
discussion of (semiparametric) efficiency of estimators is included in appendix~\ref{app:Efficiency}.

In what follows the random variables and their realized values are	both denoted by capital letters such as $X$ and $Y$. For any vector
$v\in\mathbb{R}^q$, let $v(j)$ denote the $j$-th coordinate of $v$ for
$1\le j\le q$. For any real-valued function $f(\cdot)$,
$\argmin_x f(x)$ denotes the set of all (global) minimizers of
$f(\cdot)$ and the statement $$x^* := \argmin_x f(x),$$ should be	understood as stating $x^*$ is any element of the set of all	minimizers of $f(\cdot)$. Throughout this article, the symbol $C$ is	used to denote a universal constant that can be different in different	contexts. For matrices $A, B$, $A \preceq B$ is used mean that $B - A$ is a positive semi-definite matrix.

\section{Target of Estimation}\label{sec:Target}

Suppose
$(X_i^{\top}, Y_i)^{\top}\in\mathbb{R}^p\times\mathbb{R}, 1\le i\le n$
are random vectors obtained from $n$ cases under study.  A linear	regression is performed on this data and assuming invertibility of the	matrix involved, the estimator of the ``slope'' $\hat{\beta}_n$ is	given by
\begin{equation} \label{eq:OLS-estimate}
\hat{\beta}_n = \left(\frac{1}{n}\sum_{i=1}^n X_iX_i^{\top}\right)^{-1}\left(\frac{1}{n}\sum_{i=1}^n X_iY_i\right).
\end{equation}
It is readily seen that $\hat{\beta}_n$ is a function of two averages:
one is a matrix average and the other is a vector average. For	notational convenience, let
\begin{align*}
\hat{\Sigma}_n &:= \frac{1}{n}\sum_{i=1}^n X_iX_i^{\top}\in\mathbb{R}^p\times\mathbb{R}^p;\quad\mbox{and}\quad
\hat{\Gamma}_n := \frac{1}{n}\sum_{i=1}^n X_iY_i\in\mathbb{R}^p.
\end{align*}
In the classical linear regression theory one includes the linearity	assumption $Y_i = X_i^{\top}\beta_0 + \varepsilon_i$ with
$\mathbb{E}\left[\varepsilon_i|X_i\right] = 0$.  Under this	assumption, it is easy to see that
\[
\mathbb{E}\left[\hat{\beta}_n|X_1, \ldots, X_n\right] = \beta_0\quad\Rightarrow\quad \mathbb{E}\left[\hat{\beta}_n\right] = \beta_0.
\]
Observe that independence of the observations is not required in this	calculation.  Since $\hat{\beta}_n$ is unbiased for $\beta_0$, the	estimator $\hat{\beta}_n$ can be thought of as estimating
$\beta_0$. The main question of this article is ``what is $\hat{\beta}_n$ estimating if the linearity assumption is not true?''.

As mentioned $\hat{\beta}_n$ is a function of two averages
$\hat{\Sigma}_n$ and $\hat{\Gamma}_n$. If there exist a (non-random)
matrix $\Sigma_n$ and a (non-random) vector $\Gamma_n$ such that as
$n\to\infty$
\begin{equation}\label{eq:ConsistencyAssump}
\hat{\Sigma}_n - \Sigma_n = o_p(1)\quad\mbox{and}\quad \hat{\Gamma}_n - \Gamma_n = o_p(1),
\end{equation}
then it is not unreasonable to expect that $\hat{\beta}_n$ is getting	close to
\[
\beta_n := \Sigma_n^{-1}\Gamma_n\quad\mbox{(assuming invertibility of }\Sigma_n),
\]
in the sense that $\hat{\beta}_n - \beta_n = o_p(1)$.  Indeed, this	can be easily formalized by Slutsky's theorem.  There are many cases where	assumption \eqref{eq:ConsistencyAssump} holds true and some of these	are listed below.

\begin{itemize}

\item If, for all $1\le i\neq j\le n$ and $1 \le l,m \le p$, the random vectors satisfy 
\begin{align}
\mbox{Var}(X_i(l)X_i(m)) < \infty\quad&\mbox{and}\quad\mbox{Var}(X_i(l)Y_i) < \infty;\nonumber\\
\mbox{Cov}\left(X_i(l)X_i(m), X_j(l)X_j(m)\right) \le 0\quad&\mbox{and}\quad \mbox{Cov}\left(X_i(l)Y_i, X_j(l)Y_j\right) \le 0,\label{eq:NegativeAssoc}
\end{align}
and $p$ is fixed (not changing with $n$) then the random vectors	satisfy assumption \eqref{eq:ConsistencyAssump} with
\[
\Sigma_n := \frac{1}{n}\sum_{i=1}^n \mathbb{E}\left[X_iX_i^{\top}\right]\quad\mbox{and}\quad \Gamma_n := \frac{1}{n}\sum_{i=1}^n \mathbb{E}\left[X_iY_i\right].
\] 
[A special case in \eqref{eq:NegativeAssoc} is when the	observations are independent of each other. See \cite{Shao00} for a	generalization of this condition.]
The proof follows by proving that $\hat{\Sigma}_n - \Sigma_n$ and
$\hat{\Gamma}_n - \Gamma_n$ converge coordinate-wise in probability to	zero.  Since $p$ is fixed and does not change with $n$, it follows	that they also converge to zero in any norm.  The coordinate-wise	convergence in probability can be shown by directly calculating the	variance and proving that it converges to zero.  See Theorem 2.2.1 of
\cite{Durrett10}.  Assumption \eqref{eq:NegativeAssoc} essentially	declares that the observations are ``negatively associated''.

\item If the random vectors $(X_i^{\top}, Y_i)^{\top}$ are independent
  of each other and satisfy
\[
\max_{1\le i\le n} \mathbb{E}\left[\left|X_i(l)X_i(m)\right|^{1 + \delta}\right] < \infty\quad\mbox{for all}\quad 1\le l,m \le p,
\]
with $p$ fixed, then the random vectors satisfy assumption
\eqref{eq:ConsistencyAssump} with
$\Sigma_n = \mathbb{E}[\hat{\Sigma}_n]$ and
$\Gamma_n = \mathbb{E}[\hat{\Gamma}_n]$.  The proof follows by using	Theorem 3.7 and Corollary 3.9 of \cite{White2001}.  The point of this	example is that, under the independence assumption, boundedness of the	fourth moment of $X_i(l)$ is not required; it can be reduced to
$(1 + \delta)$-th moment of $X_i^2(l)$.
\end{itemize} 
Based on these calculations, we define the target of	estimation as follows.

\begin{defi} \label{def:consistency}
If the random vectors are distributed in such a way that $\hat{\beta}_n$ satisfies
\[
\|\hat{\beta}_n - \beta_n\|_2 = o_p(1)\quad\mbox{as}\quad n\to\infty,
\]
for some vector $\beta_n$, then we say $\hat{\beta}_n$ is estimating $\beta_n$ and the vector $\beta_n$ is called the \textbf{target of estimation}.
\end{defi}
\begin{rem}\emph{
  In classical mathematical statistics, one has a target of inference
  (or a parameter of interest) in mind and the goal is to estimate
  that parameter.  In contrast, we start here with the estimator and
  analyze what it is estimating -- which is then assigned as target of
  estimation.  This process is what we call an ``upside down
  analysis''.  This approach is also similar in spirit to the thinking
  in machine learning where the method of computation is introduced
  first rather than a model.  A treatment similar to the one above can
  be found in Chapter 3 of~\cite{Potscher97}.
}\end{rem}
\begin{rem}\emph{
  The target of estimation $\beta_n$ is allowed to depend on $n, p$
  and so can change when $n$ (or $p$) is increased.  Because of this
  feature, $\beta_n$ might sometimes be referred to as a ``moving
  target''.  Just from the definition above, $\beta_n$ is not unique
  in that one can always add a small constant (converging to zero) and
  that vector can still be called the target of estimation.  In all
  the cases to be dealt with, the choice of
  the target of estimation will be clear and taking any of the
  equivalent ones does not change the story.  Also, it is not required
  that $\{\beta_n\}$ as a sequence of non-random vectors converges to
  some (non-random) vector.  
}\end{rem}

\begin{rem}\emph{
  The choice of the Euclidean norm in the Definition~\eqref{def:consistency} is only for
  concreteness and can replaced by any other norm depending on the
  context.  The choice of norm only matters in so far as consistency
  in the sense of Definition \ref{def:consistency} can be proven for
  some norms and not for others.  This may be an issue when one allows
  $p$ to grow at certain rates as a function of~$n$.
}\end{rem}

The example settings and the calculations above have shown that the	target of estimation is well-defined for linear regression in many	cases.  There is, however, nothing special about linear regression and	the target of estimation can be easily derived for a large class of	estimators (possibly inspired by a very different distributional model	for the response).  Note that the least squares estimator can also be	defined as
\[
\hat{\beta} = \argmin_{\theta\in\mathbb{R}^p}\,\frac{1}{n}\sum_{i=1}^n (Y_i - X_i^{\top}\theta)^2.
\]
The target of estimation in our example setting can be written as
\[
\beta_n = \argmin_{\theta\in\mathbb{R}^p}\,\frac{1}{n}\sum_{i=1}^n \mathbb{E}\left[(Y_i - X_i^{\top}\theta)^2\right].
\]
What is noteworthy in this representation is that the empirical	objective function (based on the observations) got replaced by its	expected value (or more generally the limit of the empirical	objective). This is a pattern that holds in general problems. To	elaborate, suppose $Z_i\in\mathbb{R}^q, 1\le i\le n$ are random	vectors obtained from~$n$ cases under study and the estimator
$\hat{\theta}_n$ obtained by solving the minimization problem
\begin{equation}\label{eq:MEstimator}
\hat{\theta}_n := \argmin_{\theta\in\mathbb{R}^k}\,\frac{1}{n}\sum_{i=1}^n \rho(Z_i, \theta),
\end{equation}
is considered for some (loss) function $\rho:\mathbb{R}^q\times\mathbb{R}^k\to\mathbb{R}$. Then under mild conditions it can be proved that the target of estimation for $\hat{\theta}_n$ is $\theta_n$ given by the minimization problem
\begin{equation}\label{eq:MTarget}
\theta_n := \argmin_{\theta\in\mathbb{R}^k}\,\frac{1}{n}\sum_{i=1}^n \mathbb{E}\left[\rho(Z_i, \theta)\right].
\end{equation}
This kind of optimization is called an $M$-estimation problem. We refer to~\cite{Yuan98} and~\cite{VdvW96} for details on $M$-estimators.
For the rest of this article, we continue with linear regression since it has a simple	estimator that is known in closed form and hence many properties are	easier to analyze.  It should, however, be understood that most of the	techniques here do generalize to arbitrary $M$-estimation problems.

Even though the target of estimation can be derived under minimal assumptions, it is not possible to derive asymptotic normality or in particular the asymptotic variance without introducing specific dependence structure on the observations. In Section~\ref{sec:DeteministicInequality}, we present a deterministic inequality that can be used to prove asymptotic normality in general contexts and in the remaining part of the paper for simplicity, we focus only on the case of independent observations. In the two sections to follow the problem of linear regression is considered under two settings:
\begin{enumerate} \itemsep 0em
\item independent random vectors with fixed covariates; and
\item independent and identically distributed random vectors.
\end{enumerate}
We provide only a preliminary analysis, and a more complete study is	considered in the unified framework of Section~\ref{sec:Unified} which	includes both these settings as special cases.  One of the main	ingredients in this analysis is the multidimensional Berry-Esseen	bound from~\cite{Bent04}.

\begin{thm}[Berry-Esseen Bound; Theorem 1.1 of
  \cite{Bent04}]\label{thm:MultiBEBoundBentkus}
  Suppose $W_1, \ldots, W_n$ are independent mean zero random vectors
  in $\mathbb{R}^d$.  Then there exists a universal constant $C > 0$
  such that
\[
\sup_{A\in\mathcal{C}_d}\left|\mathbb{P}\left(\frac{1}{\sqrt{n}}\sum_{i=1}^n W_i \in A\right) - \mathbb{P}\left(N\left(0, \Upsilon_n\right)\in A\right)\right| \le C\frac{d^{1/4}}{n^{1/2}}\left(\frac{1}{n}\sum_{i=1}^n \mathbb{E}\left[\norm{\Upsilon_n^{-1/2}W_k}_2^3\right]\right),
\]
where $\mathcal{C}_d$ denotes the set of all convex subsets of $\mathbb{R}^d$ and
$
\Upsilon_n := \sum_{i=1}^n \mbox{Var}\left(W_i\right)/n.$
\end{thm}

\section{Linear Regression with Fixed Covariates}\label{sec:FixedX}

In this section, we consider the problem of linear regression under	the assumption that the covariates are fixed (non-random) constants.
As mentioned before, this is one of the classical assumptions related	to linear and generalized linear models. For simplicity, let the	observations be denoted by
$(x_i^{\top}, Y_i)^{\top}\in\mathbb{R}^p\times\mathbb{R},\, 1\le i\le n$.
The covariates are written in lower case to emphasize that they are	fixed. And we assume the $Y_i$'s are independent random variables. The	least squares estimator is given by
\[
\hat{\beta}_n := \argmin_{\theta\in\mathbb{R}^p} \frac{1}{n}\sum_{i=1}^n \left\{Y_i - x_i^{\top}\theta\right\}^2 = \left(\frac{1}{n}\sum_{i=1}^n x_ix_i^{\top}\right)^{-1}\left(\frac{1}{n}\sum_{i=1}^n x_iY_i\right).
\]
The target of estimation is then given by
\begin{equation} \label{eq:pop-normal-eq}
\beta_n := \argmin_{\theta\in\mathbb{R}^p}\frac{1}{n}\sum_{i=1}^n \mathbb{E}\left[(Y_i - x_i^{\top}\theta)^2\right] = \left(\frac{1}{n}\sum_{i=1}^n x_ix_i^{\top}\right)^{-1}\left(\frac{1}{n}\sum_{i=1}^n x_i\mathbb{E}\left[Y_i\right]\right),
\end{equation}
where the expectation is taken with respect to the measure of
$Y_i$.  For simplicity let
\begin{equation} \label{eq:def-mu-sig}
\mu_i := \mathbb{E}\left[Y_i\right]\quad\mbox{and}\quad\Sigma_n := \frac{1}{n}\sum_{i=1}^n x_ix_i^{\top}.
\end{equation}
Note that because of fixed covariates we have
$\hat{\Sigma}_n = \Sigma_n$ for all $n$.  It follows that
\begin{align}
\sqrt{n}(\hat{\beta}_n - \beta_n) &= \frac{1}{\sqrt{n}}\sum_{i=1}^n \Sigma_n^{-1}x_i\left[Y_i - \mu_i\right]\label{eq:firsteq}\\ 
&= \frac{1}{\sqrt{n}}\sum_{i=1}^n \Sigma_n^{-1}x_i\left[Y_i - x_i^{\top}\beta_n\right].\label{eq:seceq}
\end{align}
The first equation \eqref{eq:firsteq} is specific to the fixed	covariate setting, while the second equation \eqref{eq:seceq} is valid	irrespective of whether $x_i$'s are fixed or random.  It is also	important to note that the summands in \eqref{eq:firsteq} are mean	zero while the ones in \eqref{eq:seceq} are not.  This follows from	the population normal equations obtained by differentiating the	objective function \eqref{eq:pop-normal-eq} defining the	target~$\beta_n$:
\[
\frac{1}{n}\sum_{i=1}^n \mathbb{E}\left[x_i\left(Y_i - x_i^{\top}\beta_n\right)\right] = 0.
\]
(Without any further assumptions, there is no reason for the	individual summand expectations to be zero.)

Since the covariates are fixed, it is clear that
$\sqrt{n}(\hat{\beta}_n - \beta_n)$ is a scaled average of	independent random vectors and the expectation of the average is	zero. Therefore, by the multidimensional Berry-Esseen bound (Theorem
\ref{thm:MultiBEBoundBentkus}), we obtain
\[
\sup_{A\in\mathcal{C}_p}\left|\mathbb{P}\left(\sqrt{n}(\hat{\beta}_n - \beta_n)\in A\right) - \mathbb{P}\left(Z\in A\right)\right| \le C \frac{p^{1/4}}{n^{1/2}} \gamma_n,
\]
where $\mathcal{C}_p$ is the set of all convex subsets of
$\mathbb{R}^p$, $Z$ is a Gaussian random vector with mean zero, and	the variance $\Psi_n$ given by
\begin{equation}\label{eq:AsymVar}
\Psi_n := n\mbox{Var}(\hat{\beta}_n - \beta_n) = \Sigma_n^{-1}K_n\Sigma_n^{-1},\quad K_n := \frac{1}{n}\sum_{i=1}^n x_ix_i^{\top}\mbox{Var}(Y_i).
\end{equation}
Here, $\gamma_n$ is defined as
\[
\gamma_n = \frac{1}{n} \sum_{i=1}^n \mathbb{E}\left[|Y_i - \mu_i|^3\right]\norm{\Psi_n^{-1}\Sigma_n^{-1}x_i}_2^3 
= \frac{1}{n}\sum_{i=1}^n \mathbb{E}\left[|Y_i - \mu_i|^3\right]\norm{\Sigma_nK_n^{-1}x_i}_2^3.
\]
Under certain (rate) assumptions on $p$ (that guarantee $p^{1/4}\gamma_n = o(n^{1/2})$), this implies that
\[
\sqrt{n}(\hat{\beta}_n - \beta_n)\overset{\mathcal{L}}{\approx} N\left(0, \Sigma_n^{-1}K_n\Sigma_n^{-1}\right).
\] 
We used the notation $\overset{\mathcal{L}}{\approx}$ to denote	approximation in law (or distribution).  To summarize, all we need to	assume for this asymptotic convergence result is the finiteness of the	third central moment of $Y_i$ and non-singularity of some matrices.
By comparison, classical linear regression analysis based on fixed	covariates and homoscedastic Gaussian errors requires the assumption	of linearity of the mean response in order to be valid.  In	particular, $\Sigma_n^{-1}/n$ defined in \eqref{eq:def-mu-sig} is {\em
  not} the variance of~$\hat{\beta}_n$. Since this {\em wrong} variance is reported in \verb|lm()| function of \verb|R|, one should be careful in interpreting the results.

In order to do inference using the estimator $\hat{\beta}_n$, one	should be able to estimate the asymptotic variance $\Psi_n$.  Note	that the $\Sigma_n$ factors of $\Psi_n$ are known and need not be	estimated.  All we need to estimate is $K_n$, the variance of $\sum x_iY_i/\sqrt{n}.$ 
In general, the summands $x_iY_i$ are non-identically distributed,
even if $Y_i$'s are identically distributed.  Since we do not know	their true expectations, it is {\em impossible} to estimate the	variance $K_n$.  (See Section~\ref{sec:VarEst} for details, and
\cite{Liu95} for a related problem.)  It {\em is}, however, possible	to construct a conservative estimator of $K_n$.  This construction	will be described in Section \ref{sec:Unified} (see \citet[page	492]{FAHR90}, and also \cite{Bac16} for an alternative proposal).
\begin{rem}\emph{
  The comment about impossibility of estimation of ``asymptotic''
  variance should be understood carefully.  The impossibility
  mentioned here is in the general context of fixed covariates with no
  more model assumptions than independence of observations.  In fact,
  if it is additionally assumed that $\mbox{Var}(Y_i) = \sigma^2(x_i)$
  for some continuous function $\sigma(\cdot)$, then the matrix $K_n$
  can be estimated consistently by non-parametrically estimating the
  function $\sigma(\cdot)$ (see, e.g., \cite{Alberto14}).
}\end{rem}

\section{Linear Regression with Random Covariates}\label{sec:RandomX}

Suppose we have $n$ subjects producing observations
$(X_i^{\top}, Y_i)^{\top}\in\mathbb{R}^p\times\mathbb{R},\, 1\le i\le	n$
and we apply linear regression on this data.  In this section, we	assume that these observations are random vectors that are not only	independent but also identically distributed.  Let
$(X^{\top}, Y)^{\top}$ be a generic random vector that is identically	distributed with the observations.  The least squares estimator is	still given by
\begin{equation} \label{eq:betahat-n-random-x}
\hat{\beta}_n := \argmin_{\theta\in\mathbb{R}^p} \frac{1}{n}\sum_{i=1}^n \left\{Y_i - X_i^{\top}\theta\right\}^2 = \left(\frac{1}{n}\sum_{i=1}^n X_iX_i^{\top}\right)^{-1}\left(\frac{1}{n}\sum_{i=1}^n X_iY_i\right).
\end{equation}
In this case, the target of estimation becomes
\begin{equation} \label{eq:beta-n-random-x}
\beta_n := \argmin_{\theta\in\mathbb{R}^p} \mathbb{E}\left[\left\{Y - X^{\top}\theta\right\}^2\right] = \left(\mathbb{E}\left[XX^{\top}\right]\right)^{-1}\left(\mathbb{E}\left[XY\right]\right).
\end{equation}
Note that the target $\beta_n$ does not overtly depend on $n$ because	of identical distribution of the random vectors.  We still index the	target by $n$ to have a consistent notation.  Furthermore, in theory	that follows the dimension $p$ of $\beta_n$ may be allowed to depend	on $n$, which introduces an indirect dependence of $\beta_n$ on~$n$.
For this reason all further population quantities will also be indexed	by~$n$. From definitions \eqref{eq:betahat-n-random-x} and
\eqref{eq:beta-n-random-x}, we have
\[
\sqrt{n} \, \hat{\Sigma}_n(\hat{\beta}_n - \beta_n) = \frac{1}{\sqrt{n}}\sum_{i=1}^n X_i(Y_i - X_i^{\top}\beta_n).
\]
In this case of iid random vectors it follows that the terms
$X_i(Y_i - X_i^{\top}\beta_n)$ are independent and identically	distributed random vectors with mean zero. Therefore, by the	multidimensional Berry-Esseen bound (Theorem
\ref{thm:MultiBEBoundBentkus}), it follows that
\[
\sup_{A\in\mathcal{C}_p}\left|\mathbb{P}\left(\frac{1}{\sqrt{n}}\sum_{i=1}^n X_i\left[Y_i - X_i^{\top}\beta_n\right]\in A\right) - \mathbb{P}\left(N(0, K_n)\in A\right)\right| \le C\frac{p^{1/4}}{n^{1/2}}\alpha_n,
\]
where
\[
K_n := \mathbb{E}\left[XX^{\top}(Y - X^{\top}\beta_n)^2\right]\quad\mbox{and}\quad \alpha_n := \mathbb{E}\left[\norm{K_n^{-1/2}X_i(Y_i - X_i^{\top}\beta_n)}_2^3\right].
\]
Therefore, under certain rate constraints on $p$ (that guarantees $p^{1/4}\alpha_n = o(n^{1/2})$),
$\sqrt{n}\,\hat{\Sigma}_n(\hat{\beta}_n - \beta_n)$ is approximately	normally distributed with mean zero and variance matrix $K_n$.  Since	the random vectors are assumed to be iid, under finite fourth moment	assumptions on the covariates, it follows that
\[
\norm{\hat{\Sigma}_n - \Sigma_n}_{op} = o_p(1),\quad\mbox{where}\quad \Sigma_n := \mathbb{E}\left[XX^{\top}\right].
\]
See \cite{Ver12} for more details related to the exact rate of this	convergence when $p/n = o(1)$. Also, see Section 4 of \cite{Uniform:Kuch18} for general results under exponential tail assumption on the observations.
 Thus, by Slutsky's theorem it follows	that
\[
\sqrt{n} (\hat{\beta}_n - \beta_n)
~\overset{\mathcal{L}}{\approx} ~
N\left(0, \Sigma_n^{-1}K_n\Sigma_n^{-1}\right),
\]
where we used the notation $\overset{\mathcal{L}}{\approx}$ for	approximation in law (or distribution) as in the previous section.
Again, for inference about $\beta_n$ using the estimator
$\hat{\beta}_n$, one needs to estimate $\Sigma_n$ and $K_n$.  The	matrix $\Sigma_n$ can be estimated readily by $\hat{\Sigma}_n$, but,
to estimate $K_n$, recall that one needs the variance of
\[
\frac{1}{\sqrt{n}}\sum_{i=1}^n X_i(Y_i - X_i^{\top}\beta_n).
\]
Because this is just a scaled average of $n$ independent identically	distributed random vectors with mean zero, $K_n$ can be consistently	estimated by
\[
\hat{K}_n := \frac{1}{n}\sum_{i=1}^n X_iX_i^{\top}(Y_i - X_i^{\top}\hat{\beta}_n)^2.
\]
To show that $\hat{K}_n$ is consistent for $K_n$, one can use the fact	that $\hat{\beta}_n$ is consistent for $\beta_n$ (see	Section~\ref{sec:VarEst} for more details). Thus, a consistent	estimator of the asymptotic variance of
$\sqrt{n}(\hat{\beta}_n - \beta_n)$ is given by
\begin{equation} \label{eq:sandwich-form}
\left(\frac{1}{n}\sum_{i=1}^n X_iX_i^{\top}\right)^{-1}\left(\frac{1}{n}\sum_{i=1}^n X_iX_i^{\top}(Y_i - X_i^{\top}\hat{\beta}_n)^2\right)\left(\frac{1}{n}\sum_{i=1}^n X_iX_i^{\top}\right)^{-1}.
\end{equation}
This is often referred to as the sandwich estimator of the asymptotic	variance (see Section \ref{sec:Unified} for more details).  It is	noteworthy that consistent estimation of the ``asymptotic'' variance	of $\hat{\beta}_n$ is possible under iid random vectors and is not	possible under fixed covariates without further assumptions.


\section{Unified Framework for Linear Regression}\label{sec:Unified} 

Before proceeding to unify both the settings of fixed and random	covariates, let us recall the main similarities and differences in the	analysis presented in the previous sections. First the similarities:
\begin{enumerate}
\item In the both cases, the least squares estimator $\hat{\beta}_n$
  has an ``asymptotic'' normal distribution with mean $\beta_n$, the
  ``moving target'' of estimation, and ``moving'' variance
  \[
  \frac{1}{n}\Sigma_n^{-1}K_n\Sigma_n^{-1},
  \quad\mbox{where}\quad 
  K_n := \frac{1}{n}\sum_{i=1}^n \mbox{Var}\left(X_i(Y_i - X_i^{\top}\beta_n)\right).
  \]
  Note that the target of estimation $\beta_n$ is different in 
  the fixed and random covariate cases.
\item The ``asymptotic'' normality result does not require any more
  assumptions than independence of observations and certain moment
  restrictions such as invertibility of the second moment matrix of
  covariates and finite fourth moments of covariates.  In particular,
  the classical assumptions of linearity and homoscedastic Gaussian
  errors are not required.
\end{enumerate}
Now the differences:
\begin{enumerate}
\item The score vectors $X_i(Y_i - X_i^{\top}\beta_n)$ are independent
  in both settings but are mean zero only in the random covariate
  setting.
\item The ``asymptotic'' variance can be consistently estimated only
  in the random covariate setting and is impossible to estimate in the
  fixed covariate setting without further assumptions.
\end{enumerate}
From this discussion it is clear that the similarities hold because of	the independence assumption and the differences arise from the	additional assumption of identical distributions.  The differences do	not derive from the stochastic properties of the covariates.  To	provide a unified analysis of linear regression that covers both	settings, we propose a framework where the random vectors
$(X_i^{\top}, Y_i)^{\top}$ are independent but are allowed to be	non-identically distributed.

Formally, the observations
$(X_i^{\top}, Y_i)^{\top}\in\mathbb{R}^{p+1}, 1\le i\le n$ are	independent with possibly non-identical distributions.  This framework	is much more general than either of the two settings -- fixed or	random covariates.  It allows for some random and some fixed	covariates as well.  The least squares linear regression estimator is	still given by
\begin{equation}\label{eq:LinRegEstimator}
\hat{\beta}_n = \argmin_{\theta\in\mathbb{R}^p}\,\frac{1}{n}\sum_{i=1}^n \left(Y_i - X_i^{\top}\theta\right)^2.
\end{equation}
The target of estimation in this framework can be defined as
\begin{equation}\label{eq:LinRegTarget}
\beta_n := \argmin_{\theta\in\mathbb{R}^p}\,\frac{1}{n}\sum_{i=1}^n\mathbb{E}\left[\left(Y_i - X_i^{\top}\theta\right)^2\right].
\end{equation}
Recall the following notations:
\begin{equation}\label{eq:SigmaGamma}
\begin{split}
\hat{\Sigma}_n := \frac{1}{n}\sum_{i=1}^n X_iX_i^{\top}\in\mathbb{R}^{p\times p},\quad&\mbox{and}\quad \hat{\Gamma}_n := \frac{1}{n}\sum_{i=1}^n X_iY_i\in\mathbb{R}^{p},\\
\Sigma_n := \frac{1}{n}\sum_{i=1}^n \mathbb{E}\left[X_iX_i^{\top}\right]\in\mathbb{R}^{p\times p},\quad&\mbox{and}\quad \Gamma_n := \frac{1}{n}\sum_{i=1}^n \mathbb{E}\left[X_iY_i\right]\in\mathbb{R}^{p}.
\end{split}
\end{equation}
Using these matrices and vectors, the estimator and the target defined	in \eqref{eq:LinRegEstimator} and \eqref{eq:LinRegTarget} can be	rewritten as
\begin{equation}\label{eq:ModifiedObjectives}
\begin{split}
\hat{\beta}_n &= \argmin_{\theta\in\mathbb{R}^p}\,\theta^{\top}\hat{\Sigma}_n\theta - \theta^{\top}\hat{\Gamma}_n,\qquad\mbox{and}\qquad
\beta_n = \argmin_{\theta\in\mathbb{R}^p}\,\theta^{\top}\Sigma_n\theta - \theta^{\top}\Gamma_n.
\end{split}
\end{equation}
Since these two objective functions are convex quadratic functions,
the minimizers can be obtained as zeros of the derivative, proving	that the estimator $\hat{\beta}_n$ satisfies
\begin{equation}\label{eq:LinRegEstimatorScore}
\hat{\Sigma}_n\hat{\beta}_n - \hat{\Gamma}_n = 0,
\end{equation}
and the target $\beta_n$ satisfies
\begin{equation}\label{eq:LinRegTargetScore}
\Sigma_n\beta_n - \Gamma_n = 0.
\end{equation}
Adding and subtracting $\beta_n$ from $\hat{\beta}_n$ in Equation
\eqref{eq:LinRegEstimatorScore} implies
\begin{equation}\label{eq:BeforeLinearRepresentation}
\hat{\Sigma}_n(\hat{\beta}_n - \beta_n) = \hat{\Gamma}_n - \hat{\Sigma}_n\beta_n,
\end{equation}
where the right hand side has zero expectation because of
\eqref{eq:LinRegTargetScore}.  Expanding the terms shows that
\[
\sqrt{n}\,\hat{\Sigma}_n(\hat{\beta}_n - \beta_n) = \frac{1}{\sqrt{n}}\sum_{i=1}^n X_i\left(Y_i - X_i^{\top}\beta_n\right) = \frac{1}{\sqrt{n}}\sum_{i=1}^n S_i, 
\]
where $S_i$ denotes the score given by
\begin{equation} \label{eq:score}
S_i ~:=~ X_i(Y_i - X_i^{\top}\beta_n) - \mathbb{E}\left[X_i(Y_i - X_i^{\top}\beta_n)\right].
\end{equation}
By the multivariate Berry-Esseen bound (Theorem
\ref{thm:MultiBEBoundBentkus}), it follows that
\begin{equation}\label{eq:BEBoundLinearRegUnified}
\sup_{A\in\mathcal{C}_p}\left|\mathbb{P}\left(\sqrt{n}\hat{\Sigma}_n(\hat{\beta}_n - \beta_n)\in A\right) - \mathbb{P}\left(N(0, K_n)\in A\right)\right| \le C\frac{p^{1/4}}{n^{1/2}}\alpha_n,
\end{equation}
where
\[
\alpha_n := \frac{1}{n}\sum_{i=1}^n \mathbb{E}\left[\norm{K_n^{-1/2}S_i}_2^3\right],\quad\mbox{and}\quad K_n = \mbox{Var}\left(\frac{1}{\sqrt{n}}\sum_{i=1}^n S_i\right) = \frac{1}{n}\sum_{i=1}^n \mathbb{E}\left[S_iS_i^{\top}\right].
\]
This Berry-Esseen bound proves that
\[
\sqrt{n}\,\hat{\Sigma}_n(\hat{\beta}_n - \beta_n) \overset{\mathcal{L}}{\approx} N\left(0, K_n\right).
\]
And since $\norm{\hat{\Sigma}_n - \Sigma_n}_{op} = o_p(1)$ as $n\to\infty$, it follows that
\[
\sqrt{n}(\hat{\beta}_n - \beta_n)\overset{\mathcal{L}}{\approx} N\left(0, \Sigma_n^{-1}K_n\Sigma_n^{-1}\right).
\]
Formally, we have proved the following theorem. 

\begin{thm}\label{thm:AsymNorm}
If $p$ is fixed (not depending on $n$), $\mathbb{E}\left[\norm{X_i}_2^6 + Y_i^6\right] \le B < \infty$ for all $i\ge 1$, and $K_n$ is invertible, then
\[
\sqrt{n}\,K_n^{-1/2}\hat{\Sigma}_n(\hat{\beta}_n - \beta_n) \overset{\mathcal{L}}{\to} N\left(0, I_p\right).
\]
Here $I_p$ denotes the identity matrix of dimension $p$.
\end{thm}


This completes the ``asymptotic'' study of linear regression estimator
$\hat{\beta}_n$ in the unified framework.  We write ``asymptotic''
because the normal approximations are actually non-asymptotic.

\begin{rem}\emph{{ (designation of covariates and response)}
  It should be clear from the discussion throughout that singling out
  a response variable $Y_i$ is arbitrary in principle and
  context-dependent in practice.  It is up to the analyst to decide
  which variables should be treated as covariates/regressors and which
  is to be treated as the response. 
}\end{rem}


\section{Variance Estimation and Bootstrap in Unified Framework}\label{sec:VarEst}

\subsection{Sandwich Variance Estimation}

The ``moving asymptotic'' variance of $\sqrt{n}(\hat{\beta}_n -
\beta_n)$, as shown in Theorem \ref{thm:AsymNorm}, is given by
$\Sigma_n^{-1}K_n\Sigma_n^{-1}$.  The $\Sigma_n$-part can be readily	estimated by $\hat{\Sigma}_n$ and the only part still in need of	estimation is $K_n$.  Recall that
\begin{align*}
K_n &= \mbox{Var}\left(\frac{1}{\sqrt{n}}\sum_{i=1}^n X_i\left(Y_i - X_i^{\top}\beta_n\right)\right)\\
&= \frac{1}{n}\sum_{i=1}^n \mathbb{E}\left[X_iX_i^{\top}\left(Y_i - X_i^{\top}\beta_n\right)^2\right]\\ &\qquad- \frac{1}{n}\sum_{i=1}^n \left(\mathbb{E}\left[X_i\left(Y_i - X_i^{\top}\beta_n\right)\right]\right)\left(\mathbb{E}\left[X_i\left(Y_i - X_i^{\top}\beta_n\right)\right]\right)^{\top}.
\end{align*}
So, $K_n$ is the variance of a scaled average of non-identically	distributed independent random vectors.  We prove in Lemma
\ref{lem:VarConsis} that such a variance cannot be estimated	consistently without further assumptions.  Accepting this for the	moment, note that
\[
K_n \preceq K_n^*,\quad\mbox{where}\quad K_n^* := \frac{1}{n}\sum_{i=1}^n \mathbb{E}\left[X_iX_i^{\top}\left(Y_i - X_i^{\top}\beta_n\right)^2\right],
\]
and the matrix $K_n^*$ can be consistently estimated by
\[
\check{K}_n := \frac{1}{n}\sum_{i=1}^n X_iX_i^{\top}\left(Y_i - X_i^{\top}\hat{\beta}_n\right)^2.
\]
Hence a conservative estimator of $K_n$ does exist and one such is	given by $\check{K}_n$.  (The notation $\check{\cdot}$ is used instead	of $\hat{\cdot}$ to emphasize that this is a conservative estimator	and not a consistent one.)  This provides a conservative estimator of	the asymptotic variance as 
\begin{equation}\label{eq:SandwichEstimator}
\left(\frac{1}{n}\sum_{i=1}^n X_iX_i^{\top}\right)^{-1}\left(\frac{1}{n}\sum_{i=1}^n X_iX_i^{\top}(Y_i - X_i^{\top}\hat{\beta}_n)^2\right)\left(\frac{1}{n}\sum_{i=1}^n X_iX_i^{\top}\right)^{-1}.
\end{equation}
This is the same as the sandwich estimator \eqref{eq:sandwich-form}
introduced for linear regression with iid random vectors.  However, it	is important to realize that in the setting of iid random observations	this is a consistent estimator, whereas in the unified framework it is	only a conservative estimator.

In the following we prove consistency of $\check{K}_n$ for $K_n^*$. For this, define an intermediate (unattainable) estimator
\[
\bar{K}_n := \frac{1}{n}\sum_{i=1}^n X_iX_i^{\top}\left(Y_i - X_i^{\top}\beta_n\right)^2.
\]
This is an average of independent random matrices that is unbiased for $K_n^*$. Hence, under the assumptions of Theorem~\ref{thm:AsymNorm}, by the results of \cite{Ver12},
\[
\norm{\bar{K}_n - K_n^*}_{op} = o_p(1).
\]
It now suffices to show that $\check{K}_n - \bar{K}_n$ converges to zero in terms of the operator norm in probability. Observe that
\begin{equation}
\begin{split}
\check{K}_n - \bar{K}_n &= \frac{2}{n}\sum_{i=1}^n X_iX_i^{\top}\left[X_i^{\top}\hat{\beta}_n - X_i^{\top}\beta_n\right] + \frac{1}{n}\sum_{i=1}^n X_iX_i^{\top}\left[X_i^{\top}\hat{\beta}_n - X_i^{\top}\beta_n\right]^2.
\end{split}
\end{equation}
Taking operator norm on both sides, we get
\begin{align*}
&\norm{\check{K}_n - \bar{K}_n}_{op}\\ 
\quad&\le \frac{2}{n}\sum_{i=1}^n \norm{X_i}_2^2\left|Y_i - X_i^{\top}\beta_n\right|\left|X_i^{\top}(\hat{\beta}_n - \beta_n)\right|\\
\quad&\quad + \frac{1}{n}\sum_{i=1}^n \norm{X_i}_2^2 \left|X_i^{\top}(\hat{\beta}_n - \beta_n)\right|^2\\
\quad&\le \left(1 + 2\left(\frac{1}{n}\sum_{i=1}^n \norm{X_i}_2^2\left|Y_i - X_i^{\top}\beta_n\right|^2\right)^{1/2}\right)\left(\frac{1}{n}\sum_{i=1}^n \norm{X_i}_2^2 \left|X_i^{\top}(\hat{\beta}_n - \beta_n)\right|^2\right).
\end{align*}
Under the assumptions of Theorem~\ref{thm:AsymNorm}, the first term above is $O_p(1)$ and the second term is converging to zero. Therefore, $\check{K}_n - K_n^*$ converges in probability to zero in terms of the operator norm. 

\begin{rem}\emph{{ (Best Conservative Estimator)}
  We have exhibited one conservative estimator for the ``moving
  asymptotic'' variance of $\hat{\beta}_n$, but many other
  conservative estimators exist, an example being the (delete-one)
  jackknife; see \cite{Long00} for more details.  It
  would be interesting to study the question of what comes closest to
  the true ``asymptotic'' variance, but we do not know of an answer at
  present.  An interesting feature of the conservative estimator
  \eqref{eq:SandwichEstimator} is that it is consistent in the case of
  iid observations, but the jackknife estimator is known to be
  (asymptotically) conservative.
}\end{rem}

\medskip

The following lemma proves that there does not exist a consistent	estimator for the variance of an average of non-identically	distributed independent random vectors.  The lemma is stated for	real-valued random variables which implies the result for random	vectors by taking one-dimensional projections.  See Proposition 3.5 of \cite{Bac16} for a	related result.

\begin{lem}\label{lem:VarConsis}
  Suppose $W_1, \ldots, W_n$ are independent random variables with
  $\mathbb{E}\left[W_i\right] = \mu_i$ and
  $\mbox{Var}(W_i) = \sigma_i^2$.  Then there does not exist a
  consistent estimator for $\eta_n^2$, where
\[
\eta_n^2 := \frac{1}{n}\sum_{i=1}^n \sigma_i^2.
\]
\end{lem}
\begin{proof}
  We need to prove that there does not exist a sequence of measurable
  functions $\{f_n(W_1, W_2, \ldots, W_n)\}$ such that as
  $n\to\infty$,
  \[
  f_n(W_1, W_2, \ldots, W_n) - \eta_n^2 \overset{P}{\to} 0,
  \]
  for arbitrary $\{(\mu_i, \sigma_i^2):\,1\le i\le n\}$.  Assuming
  that such a sequence exists, we obtain from consistency in the
  special case $\sigma_i^2 = 0$ for $i \ge 1$ that
  \begin{equation}\label{eq:Consistent}
    f_n(\mu_1, \mu_2, \ldots, \mu_n) \overset{P}{\to}0,\quad\mbox{as}\quad n\to\infty,
  \end{equation}
  for any fixed sequence $(\mu_i)_{i\ge 1}$.
  Now, fix $\varepsilon > 0$ and define the sequence of (measurable)
  sets
  \[
  A_n = \{|f_n(W_1, W_2, \ldots, W_n)| \le \varepsilon\}.
  \]
  Using \eqref{eq:Consistent}, we have that for any sequence
  $(w_i)_{i\ge 1}$ as $n\to\infty$
  \[
  \mathbb{P}\left(A_n\big|W_1 = w_1, \ldots, W_n = w_n\right) =
  \mathbbm{1}{\{|f_n(w_1, w_2, \ldots, w_n)| \le \varepsilon\}} \to 1.
  \]
  Thus by bounded convergence theorem, $\mathbb{P}(A_n) \to 1$ as $n\to\infty$. This implies that as
  $n\to\infty$,
  \[
  f_n(W_1, W_2, \ldots, W_n) \overset{P}{\to} 0,
  \]
  irrespective of what the true $\eta_n^2$ is. This contradicts the
  existence of a sequence consistent for $\eta_n^2.$
\end{proof}

\begin{rem}\emph{
The proof also implies that there is no other option than to	over-estimate the variance, if at all possible.
}\end{rem}

\subsection{From Sandwich to Bootstrap Estimators}

The sandwich estimator presented in \eqref{eq:SandwichEstimator} is a	direct or closed-form estimator of standard error (squared).  It would	be of interest to understand how various versions of bootstrap work	for the purpose of variance estimation or distributional	approximation.  In what follows we consider two different bootstrap	approaches in the unified framework.  These are different from the	residual bootstrap and the nonparametric pairs bootstrap considered in	the literature on linear regression.  See \cite{Freedman81} and
\cite{Buja14} for more details.  There are two reasons for this	different approach we take.  Firstly, the residual bootstrap isn't	applicable because it assumes linearity and iid errors.  Secondly, the	pairs or $x$-$y$ bootstrap can lead to singular linear systems in	simulations.  The bootstrap approaches provided here are applicable in	the unified framework and bypass the problem of singular linear	systems.  We call this bootstrap methodology the ``score bootstrap''
since it is based on resampling scores.  This idea was introduced and	studied under classical model assumptions in \cite{Hu2000}.

\subsection{Multiplier Score Bootstrap}

Let $W_1, W_2, \ldots, W_n$ be independent random variables that are	in turn independent of $(X_i, Y_i)$ and satisfy
\[
\mathbb{E}\left[W_i\right] = 0,\quad\mathbb{E}\left[W_i^2\right] = 1,\quad\mbox{and}\quad \mathbb{E}\left[|W_i|^3\right] < \infty.
\]
These variables need not be identically distributed but there is no	special reason for them to be non-identically distributed except for	allowing generality.  Recall that
\[
\sqrt{n}\,\hat{\Sigma}_n(\hat{\beta}_n - \beta_n) 
~=~ 
\frac{1}{\sqrt{n}}\sum_{i=1}^n X_i(Y_i - X_i^{\top}\beta_n).
\]
Define the estimated score vectors
\[
\hat{S}_i 
~=~ 
X_i(Y_i - X_i^{\top}\hat{\beta}_n) ,
\]
and observe that $\sum_{i=1}^n \hat{S}_i = 0$, which is	just the normal equations satisfied by $\hat{\beta}_n$.  Set
\begin{equation}\label{eq:DefTn-Bootstrap}
T_n ~:=~ \frac{1}{\sqrt{n}}\sum_{i=1}^n S_i\quad\mbox{and}\quad T_n^* ~:=~ \frac{1}{\sqrt{n}}\sum_{i=1}^n W_i\hat{S}_i,
\end{equation}
where $S_i$ are the true scores defined in \eqref{eq:score}.
Conditional on $\mathcal{Z}_n := \{(X_i, Y_i), 1\le i\le n\}$, $T_n^*$
is approximately normally distributed with mean zero and variance
$\check{K}_n$ and more precisely,
\begin{equation}\label{eq:OneBootstrapBE}
\sup_{A\in\mathcal{C}_p}\left|\mathbb{P}\left(T_n^*\in A\big|\mathcal{Z}_n\right) - \mathbb{P}\left(N(0, \check{K}_n)\in A\big|\mathcal{Z}_n\right)\right| 
~\le~
  C \, \frac{p^{1/4}}{n^{1/2}} \, 
  \frac{1}{n}\sum_{i=1}^n \norm{\check{K}_n^{-1/2}\hat{S}_i}_2^3 ~
  \mathbb{E}\left[|W_i|^3\right],
\end{equation}
Note that if $W_i\sim N(0, 1)$, then the distributional approximation error in~\eqref{eq:OneBootstrapBE} is exactly zero; this property makes the Gaussian choice for weights attractive in practice for finite sample performance. In this case, the multiplier bootstrap is called the Gaussian multiplier bootstrap in~\cite{Cher13}. 

As shown before $\norm{\check{K}_n - K_n^*}_{op} = o_p(1)$, and so, as
$n\to\infty$,
\begin{equation}\label{eq:ConsistencyKn}
\begin{split}
\sup_{A\in\mathcal{C}_p}\left|\mathbb{P}\left(N(0, \check{K}_n) \in A\big|\mathcal{Z}_n\right) - \mathbb{P}\left(N(0, K_n^*)\in A\right)\right| 
~&\le~ p^{1/2}\norm{\left(K_n^*\right)^{-1}\check{K}_n - I_p}_{op}^{1/2}\\
&=o_p(1).
\end{split}
\end{equation}
See Chapter 2, Example 2.3 of \cite{DasGupta08} for the inequality above. Recall the Berry-Esseen bound for linear regression from~\eqref{eq:BEBoundLinearRegUnified} as
\begin{equation}\label{eq:RegressionOriginalBE}
\sup_{A\in\mathcal{C}_p}\left|\mathbb{P}\left(\sqrt{n}\,\hat{\Sigma}_n(\hat{\beta}_n - \beta_n)\in A\right) - \mathbb{P}\left(N(0, K_n)\in A\right)\right| 
~\le~
C \, \frac{p^{1/4}}{n^{1/2}}\alpha_n.
\end{equation}
To show that the multiplier score bootstrap works, we need Anderson's	Lemma.
\begin{lem}[Corollary 3, \cite{Anderson55}]
  If $\xi\sim N(0, \Sigma)$ and $A$ is any centrally symmetric convex
  set (that is, $x\in A$ implies $-x\in A$ and $A$ convex), then 
\[
\mathbb{P}\left(\xi + y \in A\right) \le \mathbb{P}\left(\xi \in A\right)\quad\mbox{for all}\quad y.
\]
\end{lem}
\noindent By Anderson's Lemma, for any centrally convex set $A$,
\[
\mathbb{P}\left(N(0, K_n^*)\in A\right) \le \mathbb{P}\left(N(0, K_n) \in A\right),
\]
and using bounds \eqref{eq:OneBootstrapBE}, \eqref{eq:ConsistencyKn} and \eqref{eq:RegressionOriginalBE}, 
we get
\begin{equation} \label{eq:conservative-inequalities}
\begin{split}
\mathbb{P}\left(T_n \in A\right) = \mathbb{P}\left(\sqrt{n}\hat{\Sigma}_n(\hat{\beta}_n - \beta_n) \in A\right) 
&= \mathbb{P}\left(N(0, K_n)\in A\right) + o(1)\\
&\ge \mathbb{P}\left(N(0, K_n^*) \in A\right) + o(1)\\
&= \mathbb{P}\left(T_n^*\in A\big|\mathcal{Z}_n\right) + o_p(1),
\end{split}
\end{equation}
for all centrally symmetric convex sets in $\mathbb{R}^p$. Recall the definitions of $T_n$ and $T_n^*$ from~\eqref{eq:DefTn-Bootstrap}. For further use rewrite inequalities~\eqref{eq:conservative-inequalities} as
\begin{equation}\label{eq:conservative-ineq2}
\inf_{A\in\bar{\mathcal{C}}_p}\left(\int_A dP_{T_n} - \int_A dP_{T_n^*|\mathcal{Z}_n}\right) \ge o_p(1).
\end{equation}
Here $P_{T_n}$ and $P_{T_n^*|\mathcal{Z}_n}$ represent the	probability measure of $T_n$ and that of $T_n^*$ conditional on
$\mathcal{Z}_n$, respectively. The $o_p(1)$ on the right hand side is	with respect to the distribution of $\mathcal{Z}_n$.

These inequalities can be used for an asymptotic justification of the	simulation-based multiplier bootstrap: Suppose we generate $B_n$ draws
$(W_1^{*b},\ldots,W_n^{*b})$ ($b=1,\ldots,B_n$), calculate the associated	bootstrap statistics $T_n^{*b}$, and construct the bootstrap empirical	measure defined by
\[
\hat{\mu}_n(A) = \frac{1}{B_n}\sum_{b = 1}^{B_n}\,\mathbbm{1}\{T_n^{*b} \in A\},\quad\mbox{for any Borel set}\quad A\subseteq\mathbb{R}^p.
\]
The measure $\hat{\mu}_n(\cdot)$ is random due to randomness in
$\mathcal{Z}_n$ and in $(W_1^{*b},...,W_n^{*b})$.  Note that
$T_n^{*b}$ are iid random vectors conditional on $\mathcal{Z}_n$.
For any Borel set $A$ we have
\[
\mathbb{E}\left[\hat{\mu}_n(A)\big|\mathcal{Z}_n\right] = \mathbb{P}\left(T_n^*\in A\big|\mathcal{Z}_n\right).
\] 
Hence for various classes of sets
$\mathcal{C}^{\star}\subseteq\mathcal{C}_p$, conditional on
$\mathcal{Z}_n$, as $B_n\to\infty$, we have
\begin{equation}\label{eq:Bootstrap-convergence}
\sup_{A\in\mathcal{C}^{\star}}\left|\hat{\mu}_n(A) - \int_A dP_{T_n^*\big|\mathcal{Z}_n}\right| 
~=~ o_p(1),
\end{equation} 
where $o_p(1)$ on the right hand side is with respect to the	distribution of bootstrap samples. The class~$\mathcal{C}^{\star}$ of	sets that satisfy~\eqref{eq:Bootstrap-convergence} are called	Glivenko-Cantelli (GC) classes. The classes of all rectangles and	ellipsoids have been shown to be GC classes. See~\cite{Elker79},
\citet[Page 75]{Devroye82} and \citet[Chapter II]{Pollard84} for more	precise results.

Combining results~\eqref{eq:conservative-ineq2} and~\eqref{eq:Bootstrap-convergence}, we obtain
\begin{equation}\label{eq:Final-bootstrap}
\inf_{A\in\mathcal{C}^{\star}}\left(\int_A dP_{T_n} - \frac{1}{B_n}\sum_{b = 1}^{B_n}\,\mathbbm{1}\{T_n^*\in A\}\right) \ge o_p(1),
\end{equation}
where $o_p(1)$ refers to both the randomness of the data $\mathcal{Z}_n$ and the randomness of the bootstrap samples. Suppose now we construct a set $\hat{\mathcal{R}}_n(\alpha)\in\mathcal{C}^{\star}$ for $\alpha\in[0,1]$ such that 
\[
\frac{1}{B_n}\sum_{b = 1}^{B_n}\,\mathbbm{1}\{T_n^*\in \hat{\mathcal{R}}_n(\alpha)\} = 1 - \alpha.
\]
Then from inequality~\eqref{eq:Final-bootstrap}, it follows that as
$n\to\infty,$
\begin{equation}\label{eq:uniform-alpha}
\inf_{\alpha\in[0,1]}\left(\int_{\hat{\mathcal{R}}_n(\alpha)}dP_{T_n} - (1 - \alpha)\right) \ge o_p(1),
\end{equation}
where the $o_p(1)$ is exactly the one	from~\eqref{eq:Final-bootstrap}. Since $\hat{\mathcal{R}}_n(\alpha)$ is	random, the integral on the left hand side is a random	quantity. Recall that
$T_n = \sqrt{n}\,\hat{\Sigma}_n(\hat{\beta}_n - \beta_n)$ and so, the	above inequality implies that the confidence region
$\hat{\mathcal{R}}_n(\alpha)$ provides an asymptotically conservative	confidence region for $\beta_n$. Note here that $\alpha$ can be chosen	based on the data and validity still holds.  

It is clear from this analysis that the multiplier score bootstrap	ends up providing inference based on the same conservative variance	estimator as the direct sandwich estimator constructed before.  We	observe that the main decision was to apply the bootstrap at the level	of scores as opposed to the original data (and OLS applied to them).
The resampling bootstrap at the level of scores would allow a similar	analysis as given above for the multiplier bootstrap, and this	will be outlined in the following subsection.

\subsection{Resampling Score Bootstrap}

We consider briefly the $m$-of-$n$ resampling bootstrap applied to the	score vectors.  The associated resampling bootstrap statistic is
\[
T_m^{\star} ~:=~ \frac{1}{\sqrt{m}}\sum_{j=1}^m \hat{S}_{I_j},
\] 
where ${I_j, 1\le j\le m}$ represents an sample of $m$ iid uniform	random variables drawn from $\{1,2,\ldots,n\}$ (i.e., sampling with	replacement).  Applying the multidimensional Berry-Esseen bound	conditional on the data $\mathcal{Z}_n$, we obtain
\begin{equation} \label{eq:resampling-score}
\sup_{A\in\mathcal{C}_p}\left|\mathbb{P}\left(T_m^{\star}\in A\big|\mathcal{Z}_n\right) - \mathbb{P}\left(N(0, \check{K}_n)\in A\right)\right| 
~\le~
C \, \frac{p^{1/4}}{m^{1/2}} \, \frac{1}{n}\sum_{i=1}^n \norm{\check{K}_n^{-1/2}\hat{S}_i}_2^3 .
\end{equation}
Now, retracing the steps of the previous subsection, we conclude that	the resampling score bootstrap also produces asymptotically	conservative inference based on the same conservative variance	estimator as the sandwich.

Note that for fixed $p$ one requires a large resampling size $m$ for	the normal approximation to be good.  If $m$ does not grow as fast as
$n$, then the bound in \eqref{eq:resampling-score} dominates the error	in the coverage of the bootstrap confidence region.


\section{Hypothesis Testing in the Unified Framework}\label{sec:tAndFTests}

In the previous sections, we considered inference based on confidence	regions.  In this section we consider inference based on hypothesis	testing.
 Consider now	the test of the hypothesis
\[
H_0: \beta_n(j) = \beta_{n,0} 
\quad\mbox{versus}\quad 
H_1: \beta_n(j) \neq \beta_{n,0},
\]
for a fixed $j\in\{1,2,\ldots, p\}$ and some fixed
$\beta_{n,0} \in \mathbb{R}$.  If $\beta_{n,0} = 0$, then this is the problem	of establishing statistical significance of the (linear) effect as	measured by the coefficient $\beta_n(j)$ of the $j$-th covariate on	the response $Y$.  The only estimator for $\beta_n$ we considered was
$\hat{\beta}_n$, and so a reasonable test can be based on
$\hat{\beta}_n(j)$.  Recall that
\[
\hat{\Sigma}_n(\hat{\beta}_n - \beta_n) = \frac{1}{n}\sum_{i=1}^n X_i(Y_i - X_i^{\top}\beta_n),
\]
where the right hand side has mean zero with summands	possibly of non-zero mean.  Since $\hat{\Sigma}_n$ is a consistent	estimator for $\Sigma_n$ in the sense that
$\|\hat{\Sigma}_n - \Sigma_n\|_{op} = o_p(1)$ as $n\to\infty$,
\begin{equation}\label{eq:AsymLinearRepresentation}
\sqrt{n}(\hat{\beta}_n - \beta_n) = \frac{1}{\sqrt{n}}\sum_{i=1}^n \left[\Sigma_n^{-1}X_i\right](Y_i - X_i^{\top}\beta_n) + o_p(1).
\end{equation}
The right hand side, by the multivariate Berry-Esseen bound, has an	approximate normal distribution with mean zero and variance matrix
\[
\mbox{AV}_n := \Sigma_n^{-1}\left(\frac{1}{n}\sum_{i=1}^n \mbox{Var}\left(X_i(Y_i - X_i^{\top}\beta_n)\right)\right)\Sigma_n^{-1}.
\]
This asymptotic normal approximation implies that for any fixed
$1\le j\le p$,
\[
\frac{\sqrt{n}\left(\hat{\beta}(j) - \beta_n(j)\right)}{\sqrt{\mbox{AV}_n(j,j)}} 
~\overset{\mathcal{L}}{\to} ~
N(0,1).
\]
Here the notation $\overset{\mathcal{L}}{\to}$ is used to denote	convergence in law (or distribution).  As proved in previous sections,
there does not exist a consistent estimator for $\mbox{AV}_n$ (in this	general framework) but there exists a (asymptotically) conservative	estimator given by
\[
\widecheck{\mbox{AV}}_n 
~:=~
\hat{\Sigma}_n^{-1}\left(\frac{1}{n}\sum_{i=1}^n X_iX_i^{\top}(Y_i - X_i^{\top}\hat{\beta}_n)^2\right)\hat{\Sigma}_n^{-1}.
\]
This is consistent for
\[
{\mbox{AV}}^*_n 
~:=~
\Sigma_n^{-1}\left(\frac{1}{n}\sum_{i=1}^n \mathbb{E}\left[X_iX_i^{\top}(Y_i - X_i^{\top}\beta_n)^2\right]\right)\Sigma_n^{-1}.
\]
Thus, by Slutsky's theorem,
\[
\frac{\sqrt{n}(\hat{\beta}_n(j) - \beta_n(j))}{\sqrt{\widecheck{\mbox{AV}}_n(j,j)}} 
~\overset{\mathcal{L}}{\approx}~
N\left(0, \frac{\mbox{AV}_n(j,j)}{\mbox{AV}^*_n(j,j)}\right).
\]
Here the variance of the normal distribution on the right is at most	1.  Since this ratio cannot be estimated consistently, one solution is	to conservatively use $N(0,1)$ instead.  To perform the test replace
$\beta_n(j)$ by $\beta_{n,0}$ and use this normal distribution.  So, the	test is based on the statistic
\[
t_j ~:=~ 
\frac{\sqrt{n}(\hat{\beta}_n(j) - \beta_{n,0})}{\sqrt{\widecheck{\mbox{AV}}_n(j,j)}}.
\]
In the classical linear regression model, the denominator for the same	hypothesis testing problem is given by the classical estimator of the	variance obtained under the assumption of correct specification.  The	test statistic $t_j$ has then a $t$-distribution.  That denominator is	not valid in the unified framework which permits misspecification.
The present statistic $t_j$ hence cannot be assumed to have a
$t$-distribution.  Note that the test based on $t_j$ leads to a	conservative test, meaning the type-I error, in this general	framework, would be strictly smaller than $\alpha$ (asymptotically).
One subtle point here is that this conservativeness does not arise	from $\mbox{AV}^*_n$ but from the use of $N(0, 1)$ instead of the	correct but unattainable normal distribution.  Because
$t$-distributions have heavier tails than $N(0, 1)$, their use would	result in additional conservativeness.  Such could be considered	desirable by those who wish to account for estimated degrees of	freedom.

Suppose now we want to simultaneously test over all $1\le j\le p$
instead of just one of them, that is,
\[
H_0 : \beta_n = \beta_{n,0} \quad \mbox{versus} \quad H_1: \beta_n \neq \beta_{n,0}
\]
for some vector $\beta_{n,0} \in \mathbb{R}^p$.  This testing problem	is usually addressed by an $F$-test, but an intuitive alternative can	be based on the ``max-$|t|$'' statistic defined by
\[
T_{n,p} := \max_{1\le j\le p}|t_j| = \max_{1\le j\le p}\left|\frac{\sqrt{n}(\hat{\beta}_n(j) - \beta_0(j))}{\sqrt{\widecheck{\mbox{AV}}_n(j,j)}}\right|.
\]
The name ``max-$|t|$'' derives from classical linear regression	theory, but in the current context of a unified framework this is	strictly speaking a misnomer.  

We end this section with one last point: Even though all the above	tests are asymptotically conservative, they may not be conservative	for inference in finite samples because of asymptotic approximation	error.
\section{Deterministic Inequality for Linear Regression}\label{sec:DeteministicInequality}
In the previous sections, the observations are assumed to be independent and based on multivariate Berry-Esseen bounds we proved asymptotic normality of the estimator. In this section, we show, in a more direct way, that as long as a version of central limit theorem exists the linear regression estimator works. 

Recall from~\eqref{eq:SigmaGamma}, the definitions of $\hat{\Sigma}_n, \Sigma_n, \hat{\Gamma}_n$ and $\Gamma_n$. The least squares estimator $\hat{\beta}_n$ and $\beta_n$ are given by
\[
\hat{\beta}_n = \hat{\Sigma}_n^{-1}\hat{\Gamma}_n,\qquad\mbox{and}\qquad \beta_n = \Sigma_n^{-1}\Gamma_n.
\]
Note that these definitions do not require any structure on the dependence or the distributions of the random vectors. Define
\begin{equation}\label{eq:ErrorNorms}
\Lambda_n := \lambda_{\min}(\Sigma_n),\qquad\mbox{and}\qquad \mathcal{D}_{2n}^{\Sigma} := \|\hat{\Sigma}_n - \Sigma_n\|_{op}.
\end{equation}
Under this setting, the following deterministic inequality holds. This result is implicitly present in the calculations of previous sections and appeared in~\cite{Uniform:Kuch18} in a general context of post-selection inference and uniform-in-model results for OLS linear regression.
\begin{thm}\label{thm:DeterIneqOLS}
If $\mathcal{D}_{2n}^{\Sigma} \le \Lambda_n/2$, then the estimator $\hat{\beta}_n$ satisfies
\begin{equation}\label{eq:EstimationError}
\frac{1}{2}\|\Sigma_n^{-1}(\hat{\Gamma}_n - \hat{\Sigma}_n\beta_n)\|_2 ~~\le~~ \|\hat{\beta}_n - \beta_n\|_2 ~~\le~~ 2\|\Sigma_n^{-1}(\hat{\Gamma}_n - \hat{\Sigma}_n\beta_n)\|_2,
\end{equation}
and
\begin{equation}\label{eq:LinearApproxError}
\norm{\hat{\beta}_n - \beta_n - \Sigma_n^{-1}(\hat{\Gamma}_n - \hat{\Sigma}_n\beta_n)}_2 ~\le~ \frac{2\mathcal{D}_{2n}^{\Sigma}\|\Sigma_n^{-1}(\hat{\Gamma}_n - \hat{\Sigma}_n\beta_n)\|_2}{\Lambda_n}.
\end{equation}
\end{thm}
\begin{proof}
From the definition of $\hat{\beta}_n$, it follows that
\[
\hat{\Sigma}_n\hat{\beta}_n = \hat{\Gamma}_n.
\]
So subtracting $\hat{\Sigma}_n\beta_n$ from both sides, we get
\[
\hat{\Sigma}_n(\hat{\beta}_n - \beta_n) = \hat{\Gamma}_n - \hat{\Sigma}_n\beta_n.
\]
Now writing $\hat{\Sigma}_n = \Sigma_n - (\Sigma_n - \hat{\Sigma}_n)$ on the left hand side, we get
\[
\Sigma_n(\hat{\beta}_n - \beta_n) = \hat{\Gamma}_n - \hat{\Sigma}_n\beta_n + \left(\Sigma_n - \hat{\Sigma}_n\right)(\hat{\beta}_n - \beta_n).
\]
Therefore, by multiplying by $\Sigma_n^{-1}$ we obtain
\[
\hat{\beta}_n - \beta_n - \Sigma_n^{-1}(\hat{\Gamma}_n - \hat{\Sigma}_n\beta_n) = \Sigma_n^{-1}\left(\Sigma_n - \hat{\Sigma}_n\right)(\hat{\beta}_n - \beta_n).
\]
Taking $\norm{\cdot}_2$-norm on both sides implies that
\begin{equation}\label{eq:LinearError}
\begin{split}
\norm{\hat{\beta}_n - \beta_n - \Sigma_n^{-1}(\hat{\Gamma}_n - \hat{\Sigma}_n\beta_n)}_2 &\le 
\frac{\|\hat{\Sigma}_n - \Sigma_n\|_{op}}{\lambda_{\min}(\Sigma_n)}\|\hat{\beta}_n - \beta_n\|_2.
\end{split}
\end{equation}
Hence,
\[
\left(1 - \frac{\mathcal{D}_{2n}^{\Sigma}}{\Lambda_{n}}\right)\|\hat{\beta}_n - \beta_n\|_2 \le \norm{\Sigma_n^{-1}\left(\hat\Gamma_n - \hat\Sigma_n\beta_n\right)}_2 \le \left(1 + \frac{\mathcal{D}_{2n}^{\Sigma}}{\Lambda_{n}}\right)\|\hat{\beta}_n - \beta_n\|_2.
\]
Therefore, if $\mathcal{D}_{2n}^{\Sigma} \le \Lambda_n/2$, then
\[
\frac{1}{2}\norm{\Sigma_n^{-1}(\hat{\Gamma}_n - \hat{\Sigma}_n\beta_n)}_2 \le \|\hat{\beta}_n - \beta_n\|_2 \le 2\norm{\Sigma_n^{-1}(\hat{\Gamma}_n - \hat{\Sigma}_n\beta_n)}_2,
\]
which proves~\eqref{eq:EstimationError}. Substituting this inequality in~\eqref{eq:LinearError}, we get
\[
\norm{\hat{\beta}_n - \beta_n - \Sigma_n^{-1}(\hat{\Gamma}_n - \hat{\Sigma}_n\beta_n)}_2 \le \frac{2\mathcal{D}_{2n}^{\Sigma}\|\Sigma_n^{-1}(\hat{\Gamma}_n - \hat{\Sigma}_n\beta_n)\|_2}{\Lambda_n},
\]
which proves~\eqref{eq:LinearApproxError}.
\end{proof}
\begin{rem}
\emph{Theorem~\ref{thm:DeterIneqOLS} is a deterministic inequality. Inequality~\eqref{eq:EstimationError} implies a \emph{necessary and sufficient} condition for $\beta_n$ to be the target of estimation for $\hat{\beta}_n$. Note that $\Sigma_n^{-1}(\hat{\Gamma}_n - \hat{\Sigma}_n\beta_n)$ is an average of random vectors the expectation of which is zero by definition of $\beta_n$. As long as this average converges to zero (a version of LLN) in $\norm{\cdot}_2$-norm, $\beta_n$ is the target of estimation for $\hat{\beta}_n$. Additionally if this average (after proper scaling) converges in distribution (a version of CLT), then $\hat{\beta}_n - \beta_n$ (after the same scaling) converges to the same distribution. Since these are based on deterministic inequalities, a combination of LLN and CLT completes the upside down analysis of the $\hat{\beta}_n$.} 
\end{rem}
\section{Conclusions on Assumptions for Linear Regression}\label{sec:Conclude}


What we find from the (essentially finite-sample) analysis in previous	sections is that we do not need any of the usual model assumptions	including linearity, normality and homoscedasticity.  Only under	independence assumptions on observations (along with some moment	assumptions), we have asymptotic normality of the LSE around its	corresponding target (properly scaled), and
\[
\left(\frac{1}{n}\sum_{i=1}^n X_iX_i^{\top}\right)^{-1}\left(\frac{1}{n}\sum_{i=1}^n X_iX_i^{\top}\left(Y_i - X_i^{\top}\hat{\beta}_n\right)^2\right)\left(\frac{1}{n}\sum_{i=1}^n X_iX_i^{\top}\right)^{-1},
\]
is an asymptotically valid estimator of the asymptotic variance of
$\sqrt{n}(\hat\beta_n - \beta_n)$.  This should be understood in the	sense that when observations are identically distributed this	estimator is consistent, and when observations are non-identically	distributed this estimator is asymptotically conservative (no	consistent estimator exists in this case).  The conservativeness in	the broader context of generalized linear models was discussed in
\citet[page 492]{FAHR90}.

In passing let us now make a comment on the assumption of independence	of observations. When discussing and defining the target of	estimation, it was shown that even the independence of observations is	not needed. To make the rates and the asymptotic distribution	concrete, the assumption of independence was introduced. Recollecting	the technical tools that went into the derivation of Theorem
\ref{thm:AsymNorm}, it can be seen that the linear representation
\eqref{eq:BeforeLinearRepresentation} (that holds without any	assumptions on the random vectors) and the multivariate Berry-Esseen	bound (Theorem \ref{thm:MultiBEBoundBentkus}) for mean zero	independent random vectors are used. So, as long as a version of a	Berry-Esseen bound or a multivariate central limit theorem exists, the	assumption of independence can be replaced by a ``weak'' dependence	assumption.  See \cite{Siegfried09} for Berry-Esseen bounds for	averages of mean zero random vectors under various dependence settings	based on an approximation with $m$-dependent sequences.  Also, see	Chapter~10 of \cite{Potscher97}. 

\section*{Acknowledgments}
The authors thank Shaokun Li for the proof of Lemma~\ref{lem:VarConsis}.

\appendix
\section{Semiparametric Efficiency}\label{app:Efficiency}
The discussion in this article was restricted to the discussion of the ``asymptotic'' properties of the estimator that the data analyst started with and was not related to how well one can estimate the target of estimation. As mentioned before, the traditional mathematical statistics was designed under correctly specified parametric models and the goal is to efficiently estimate the true parameter that determines the distribution. From the point of view of previous sections, this question in the current form does not make sense; the analyst wants to use the (least squares linear regression) estimator he/she chose irrespective of what the true model is. Alternatively, one might ask ``having chosen an estimator that leads a particular target of estimation, is there an efficient way to estimate the target of estimation?''. For example, in case the analyst has chosen to use least squares linear regression estimator, the target of estimation becomes
\[
\beta_n = \argmin_{\theta\in\mathbb{R}^p} \frac{1}{n}\sum_{i=1}^n \mathbb{E}\left[(Y_i - X_i^{\top}\theta)^2\right].
\]
What is an efficient estimator of $\beta_n$? is $\hat{\beta}_n$ an efficient estimator for $\beta_n$? what does efficiency mean here? This question naturally leads to the area of semiparametric inference and the answer exists at least in the case of iid random vectors since \cite{Levit76}. See example 5 on page 725 of \cite{Levit76}. In this appendix, we provide a heuristic argument for how should an efficient estimator look like for the case of independent observations (without identical distributions assumption). See \citet[Lectures 1-4, pages 336--382]{vdV02} and \cite{McNeney98} for ways to formalizing the result. The setting for semiparametric inference is as follows: suppose $Z_1, Z_2, \ldots, Z_n$ are $n$ independent random vectors with $Z_i\sim P_i$ for some probability distribution $P_i$ and the target of estimation is $\psi(P^{\otimes n})$ for some functional $\psi$ defined on a class of distributions $\mathcal{P}_n$ (chosen also by the analyst). Here
\[
P^{\otimes n} = \bigotimes_{i=1}^n P_i,
\]
represents the joint distribution of $(Z_1, Z_2, \ldots, Z_n)$ and $\mathcal{P}_n$ contains distributions of this type where each $P_i$ varies over some set of probability distributions. Some example might clarify the problem:
\begin{enumerate}
\item Suppose $Z_1, \ldots, Z_n$ are $n$ independent real-valued random variables and we want to estimate
\[
\psi(P^{\otimes n}) := \frac{1}{n}\sum_{i=1}^n \mathbb{E}\left[Z_i\right] = \int \left(\frac{1}{n}\sum_{i=1}^n z_i\right)dP_1(z_1)\ldots dP_n(z_n).
\]
Here $\mathcal{P}_n$ can be taken to be the set of all joint distributions of $Z_1, \ldots, Z_n$ such that the marginal variances are all uniformly bounded. One can consider the same functional with random vectors too.
\item Suppose $Z_i\in\mathbb{R}^q, 1\le i\le n$ are independent random vectors and $\rho:\mathbb{R}^q\times\mathbb{R}^k$ is some ``loss'' function. The functional to be estimated is
\[
\psi(P^{\otimes n}) := \argmin_{\theta\in\mathbb{R}^k}\,\frac{1}{n}\sum_{i=1}^n \mathbb{E}\left[\rho(Z_i, \theta)\right].
\]
Here too the class of joint distributions $\mathcal{P}_n$ can be taken to be completely nonparametric as in the previous example except for some moment restrictions to let the functional well-defined. Note that unlike the previous example, it may not be possible to explicitly write the functional in terms of $P^{\otimes n}$.
\end{enumerate}
These are called semiparametric problems since the class of all distributions is mostly nonparametric (unrestricted) and the functional of interest is Euclidean (or parametric) in nature.

The basic idea of semiparametric efficiency is as stated by \citet[Section 2]{Newey90} and \citet[Section 1.2]{vdV02}:
\begin{quote}
The semiparametric problem is at least as hard as any of the parametric problems that it encompasses.
\end{quote}
To understand this idea, briefly consider the simpler case of identical distributions so that $\mathcal{P}_n$ is a subset of the class of all joint distributions with the restriction of identical marginal distributions. Let the true distribution of observations be 
\[
\bigotimes_{i=1}^n P.
\]
As a thought experiment, think of $\mathcal{P}_n$ as constituted by joint distributions of the form
\begin{equation}\label{eq:ParametricSubmodel}
P_t^{g,\otimes n} := \bigotimes_{i=1}^n P_{t}^{(g)},
\end{equation}
for $t\in\mathbb{R}$ and $g$ varying over some class of functions, $\mathcal{G}$ with $P_{t = 0}^{(g)} = P$ for any $g\in\mathcal{G}$. So, the nature can be thought of as picking a function $g\in\mathcal{G}$ and then producing observations from $P_t^{g, \otimes n}$. If the function $g$ is known to the statistician, he/she could perform maximum likelihood estimation on the parametric (sub-)model:
\[
\mathcal{P}_n^{(g)} := \left\{\bigotimes_{i=1}^n P_t^{(g)}:\,t\in\mathbb{R}\right\},
\]
to obtain $\hat{t}$, an estimator of $t$ and then estimate the functional $\psi$ by
\begin{equation}\label{eq:MLESubmodel}
\hat{\psi}_n^{(g)} := \psi\left(P_{\hat{t}}^{g,\otimes n}\right).
\end{equation}
Under certain regularity conditions, this estimator would achieve the ``smallest'' variance asymptotically, if $g$ were known to the statistician. However, $g$ and $\mathcal{G}$ are both unknown. Hence, the statistician cannot perform better than the largest variance of $\hat{\psi}_n^{(g)}$ over $g\in\mathcal{G}$. The parametric sub-model that leads to this largest variance is called the least favorable sub-model. To use this idea, one would usually take parametric sub-models of the form \eqref{eq:ParametricSubmodel} that are contained in $\mathcal{P}_n$ and take the largest efficient variance over $g\in\mathcal{G}$ as the best possible variance in the semiparametric setting.

To see this idea in action, note first that the variance of $\hat{\psi}_n^{(g)}$ (in \eqref{eq:MLESubmodel}) asymptotically should be given by the Cramer-Rao lower bound, under regularity conditions. We recall the Cramer-Rao lower bound here with proof for completeness.
\begin{lem}[Cramer-Rao Lower Bound]\label{lem:LowerBound}
If $h(X)$ is an unbiased estimator of $\psi(P)\in\mathbb{R}$ for $P\in\{P_{\theta}:\,\theta\in\Theta\}$ absolutely continuous with respect to the Lebesgue measure for some open subset $\Theta\subseteq\mathbb{R}^k$, then under conditions allowing interchange of derivative and integral 
\[
\mbox{Var}_{P_{\theta_0}}\left(h(X)\right) \ge \frac{\left[\psi'(P_{\theta_0})\right]^2}{\mathbb{E}\left[\dot{\ell}_{\theta_0}^2(X)\right]}.
\]
Here
\[
\psi'(P_{\theta_0}) = \frac{d\psi(P_{\theta})}{d\theta}\big|_{\theta = \theta_0}\quad\mbox{and}\quad \dot{\ell}_{\theta_0}(x) = \frac{d}{d\theta}\log dP_{\theta}(x)\big|_{\theta = \theta_0}.
\]
The function $\dot{\ell}_{\theta_0}(x)$ is called the likelihood score.
\end{lem}
\begin{proof}
From the hypothesis of unbiasedness, we obtain
\[
\int h(x)dP_{\theta}(x) = \psi(P_{\theta})\quad\mbox{for all}\quad \theta\in\Theta.
\]
Now differentiating with respect to $\theta$, it follows that
\[
\int h(x)\dot{\ell}_{\theta_0}(x)dP_{\theta_0}(x) = \psi'(P_{\theta_0}).
\]
Equivalently, this can be written as
\[
\mbox{Cov}_{P_{\theta_0}}\left(h(X), \dot{\ell}_{\theta_0}(X)\right) = \psi'(P_{\theta_0}).
\]
Using Cauchy-Schwarz inequality, we get
\[
\mbox{Var}_{P_{\theta_0}}\left(h(X)\right) \ge \frac{\left[\psi'(P_{\theta_0})\right]^2}{\mathbb{E}\left[\dot{\ell}_{\theta_0}^2(X)\right]},
\]
proving the result.
\end{proof}
Now getting back to the semiparametric problem with independent but possibly non-identically distributed observations, consider the parametric sub-model,
\[
dP_t^{(g_1, \ldots, g_n)}(z_1, \ldots, z_n) := \prod_{i=1}^n c(t, g_i)K(tg_i(z_i))dP_i(z_i),\quad t\in\mathbb{R},
\]
where $g_i(\cdot), 1\le i\le n$ represent any set of $n$ functions satisfying $\int g_i(z)dP_i(z) = 0$ and $\int g_i^2(z)dP_i(z) < \infty.$ Here the function $K(\cdot)$ is given by
\[
K(u) = 2(1 + \exp(-2u))^{-1},
\]
and $c(t, g_i)$ is a positive normalizing constant. Since
\[
c(t, g_i)\int K(tg_i(z))dP_i(z) = 1\quad\mbox{for all}\quad t\in\mathbb{R},
\]
it follows that $c(0, g_i) = 1$. See example 1.12 on page 346 of \cite{vdV02}. Differentiating with respect to $t$ and taking $t = 0$ proves
\[
c'(0, g_i) = \frac{d}{dt}c(t, g_i)\big|_{t = 0} = 0.
\]
Therefore, the ``likelihood score'' term is given by 
\[
\frac{d}{dt}\log\frac{dP_t^{(g_1,\ldots, g_n)}(z_1, \ldots, z_n)}{dP^{\otimes n}(z_1, \ldots, z_n)}\bigg|_{t = 0} = \sum_{i=1}^n g_i(z_i),
\]
and
\[
\sum_{i=1}^n \int g_i(z_i)dP_i(z_i) = 0.
\]
By independence of observations,
\begin{equation}\label{eq:Information}
\mathbb{E}\left[\left(\sum_{i=1}^n g_i(Z_i)\right)^2\right] = \sum_{i=1}^n \mathbb{E}\left[g_i^2(Z_i)\right].
\end{equation}
To find the semiparametric lower bound, all we need to find is the ``derivative'' of the functional. For our purposes, all the functionals we work with are of the form given in example 2 above, that is
\[
\psi\left(P^{\otimes n}\right) := \argmin_{\theta\in\mathbb{R}^k}\,\frac{1}{n}\sum_{i=1}^n \mathbb{E}\left[\rho(Z_i, \theta)\right].
\]
We deal with the case $k = 1$ and the general case follows by taking linear combinations of the functional. Assume that $\rho(\cdot, \cdot)$ is twice differentiable with respect to the second argument and let
\[
\Psi(z, \theta) := \frac{d}{d\theta}\rho(z, \theta)\quad\mbox{and}\quad \dot{\Psi}(z, \theta) := \frac{d}{d\theta}\Psi(z, \theta).
\]
Using this differentiability, it follows that for all $P^{\otimes n}$,
\[
\sum_{i=1}^n \int \Psi\left(z_i, \psi\left(P^{\otimes n}\right)\right)dP_i(z_i) = 0.
\]
Taking $P^{\otimes n}$ to be $P_t^{(g_1, \ldots, g_n)}$, we get for all $t\in\mathbb{R}$,
\[
\sum_{i=1}^n c(t, g_i)\Psi\left(z_i, \psi\left(P_t^{(g_1, \ldots, g_n)}\right)\right)K(tg_i(z_i))dP_i(z_i) = 0.
\]
Differentiating with respect to $t$ and taking $t = 0$, it follows that
\begin{align}
\frac{d}{dt}\psi\left(P_t^{(g_1, \ldots, g_n)}\right) &= \left(\sum_{i=1}^n \mathbb{E}\left[\dot{\Psi}\left(Z_i, \psi\left(P^{\otimes n}\right)\right)\right]\right)^{-1}\mathbb{E}\left[\sum_{i=1}^n \Psi(Z_i, \psi\left(P^{\otimes n}\right))g_i(Z_i)\right]\nonumber\\ 
&= \mathbb{E}\left[\sum_{i=1}^n \tilde{\psi}_{P_i}(Z_i)g_i(Z_i)\right].\label{eq:InfluenceFunction}
\end{align}
Here 
\[
\tilde{\psi}_{P_i}(z_i) := \left(\sum_{i=1}^n \mathbb{E}\left[\dot{\Psi}\left(Z_i, \psi\left(P^{\otimes n}\right)\right)\right]\right)^{-1}\left\{\Psi(z_i, \psi\left(P^{\otimes n}\right)) - \mathbb{E}\left[\Psi(Z_i, \psi\left(P^{\otimes n}\right))\right]\right\},
\]
and the properties $c(0, g_i) = 1, c'(0, g_i) = 0, K(0) = 1, K'(0) = 1$ are used. The function $\tilde{\psi}_{P_i}(\cdot)$ is called the ``efficient influence function'' for the iid case. Substituting \eqref{eq:Information} and \eqref{eq:InfluenceFunction} in the Cramer-Rao lower bound (Lemma \ref{lem:LowerBound}) and maximizing with respect to all $g_i, 1\le i\le n$ with finite second moment implies
\[
\sup_{\substack{(g_i)_{1\le i\le n}:\\\mathbb{E}(g_i(Z_i)) = 0, \mathbb{E}(g_i^2(Z_i)) < \infty}} \left(\sum_{i=1}^n \mathbb{E}\left[g_i^2(Z_i)\right]\right)^{-1}\left(\sum_{i=1}^n \mathbb{E}\left[\tilde{\psi}_{P_i}(Z_i)g_i(Z_i)\right]\right)^2 = \mathbb{E}\left[\sum_{i=1}^n \tilde{\psi}_{P_i}^2(Z_i)\right].
\]
This maximum is attained for $g_i(z_i) = \tilde{\psi}_{P_i}(z_i)$. By a semiparametric extension of regular estimator, this implies that any regular efficient estimator $T_n$ must have an asymptotic linear representation given by
\[
\sqrt{n}\left(T_n - \psi\left(P^{\otimes n}\right)\right) = {\sqrt{n}}\sum_{i=1}^n \tilde{\psi}_{P_i}(Z_i) + o_p(1).
\]
Note that $\mathbb{E}[\tilde{\psi}_{P_i}(Z_i)] = 0$ and $\mbox{Var}(\tilde{\psi}_{P_i}(Z_i)) < \infty$. By an application of Lindeberg-Feller theorem, it follows that $T_n$ under suitable normalization has an asymptotic normal distribution under the Lindeberg condition. 

\subsection{Application to Linear Regression}
For the case of linear regression, $Z_i = (X_i, Y_i)$ and $\rho(z, \theta) = (y - x^{\top}\theta)^2$. Therefore,
\[
\Psi(z, \theta) = -x(y - x^{\top}\theta)\quad\mbox{and}\quad \dot{\Psi}(z, \theta) = -xx^{\top}.
\]
Hence, any efficient regular estimator $T_n$ of the target of estimator $\beta_n$ must have an asymptotic linear representation given by
\begin{align*}
\sqrt{n}\left(T_n - \beta_n\right) &= {\sqrt{n}}\sum_{i=1}^n \left(\sum_{i=1}^n \mathbb{E}\left[X_iX_i^{\top}\right]\right)^{-1}\left\{X_i(Y_i - X_i^{\top}\beta_n) - \mathbb{E}\left[X_i(Y_i - X_i^{\top}\beta_n)\right]\right\}\\ &\qquad+ o_p(1)\\
&= \frac{1}{\sqrt{n}}\sum_{i=1}^n \Sigma_n^{-1}\left\{X_i(Y_i - X_i^{\top}\beta_n) - \mathbb{E}\left[X_i(Y_i - X_i^{\top}\beta_n)\right]\right\} + o_p(1)\\
&= \frac{1}{\sqrt{n}}\sum_{i=1}^n \Sigma_n^{-1}X_i(Y_i - X_i^{\top}\beta_n) + o_p(1).
\end{align*}
Realize from Equation \eqref{eq:AsymLinearRepresentation} that the least squares linear regression estimator $\hat{\beta}_n$ satisfies the linear representation and so the least squares estimator is a semiparametrically efficient estimator of $\beta_n$. Similar calculations holds for $M$-estimators obtained from generalized linear models.
\bibliographystyle{apalike}
\bibliography{AssumpLean}
\end{document}